\documentclass[letterpaper, 10 pt, conference]{ieeeconf}

\IEEEoverridecommandlockouts                           

\usepackage[utf8]{inputenc}
\usepackage[T1]{fontenc}
\usepackage{xcolor}
\usepackage{graphicx}
\usepackage{xspace}

\usepackage{algorithm}
\usepackage{algpseudocode}
\algrenewcommand\algorithmicrequire{\textbf{Input:}}
\newcommand{\CommentState}[1]{\Statex\hspace{\algorithmicindent}{\color{blue}// #1}}

\usepackage{amsmath,amsfonts,amssymb}
\usepackage{bm,fixmath} 

\usepackage{tikz}
\usetikzlibrary{arrows, positioning, calc}

\newtheorem{assumption}{Assumption}
\newtheorem{lemma}{Lemma}
\newtheorem{proposition}{Proposition}
\newtheorem{remark}{Remark}

\DeclareMathOperator*{\blk}{blk\,diag}

\newcommand{\N}{\mathbb{N}}
\newcommand{\R}{\mathbb{R}}

\newcommand{\norm}[1]{\left\lVert#1\right\rVert}

\newcommand{\x}{\mathbold{x}}
\newcommand{\y}{\mathbold{y}}
\newcommand{\z}{\mathbold{z}}
\newcommand{\w}{\mathbold{w}}
\newcommand{\uv}{\mathbold{u}}

\newcommand{\W}{\mathbold{W}}

\newcommand{\0}{\mathbf{0}}
\newcommand{\1}{\mathbf{1}}

\usepackage{accents}
\newcommand{\lmax}{\bar{\lambda}}
\newcommand{\lmin}{\underline{\lambda}}

\newcommand{\ubar}[1]{\underaccent{\bar}{#1}}

\title{\LARGE \bf
Control-Based Online Distributed Optimization
}
\author{%
	Wouter J. A. van Weerelt$^{1}$, Nicola~Bastianello$^{2,\star}$
    \thanks{The work of N.B. was partially supported by the European Union’s Horizon Research and Innovation Actions programme under grant agreement No. 101070162.}
	\thanks{$^{1}$KTH Royal Institute of Technology, Stockholm, Sweden.}%
	\thanks{$^{2}$School of Electrical Engineering and Computer Science, and Digital Futures, KTH Royal Institute of Technology, Stockholm, Sweden.}
	\thanks{$^{\star}$Corresponding author. Email: {\tt\footnotesize nicolba@kth.se}.}%
}

\begin{document}

\maketitle

\thispagestyle{plain}
\pagestyle{plain}


\begin{abstract}
In this paper we design a novel class of online distributed optimization algorithms leveraging control theoretical techniques.
We start by focusing on quadratic costs, and assuming to know an internal model of their variation. In this set-up, we formulate the algorithm design as a robust control problem, showing that it yields a fully distributed algorithm. We also provide a distributed routine to acquire the internal model.
We show that the algorithm converges exactly to the sequence of optimal solutions.
We empirically evaluate the performance of the algorithm for different choices of parameters.
Additionally, we evaluate the performance of the algorithm for quadratic problems with inexact internal model and non-quadratic problems, and show that it outperforms alternative algorithms in both scenarios.
\end{abstract}

\section{Introduction}\label{sec:intro}
Cyber-physical systems, composed of interconnected smart devices, are now widespread in a range of applications, including power systems, transportation, internet-of-things \cite{molzahn_survey_2017,nedic_distributed_2018,yang_survey_2019}.
These devices utilize their sensing, communication, and computation resources to cooperatively accomplish a broad range of tasks, \textit{e.g.} learning, coordination, navigation \cite{nedic_distributed_2018}. These tasks can be encoded as distributed optimization problems, and there is thus a need for suitable solution algorithms.
However, these devices are often embedded in a dynamic environment, and the data they collect changes over time. Therefore, we are interested in solving \textit{online} distributed optimization problems:
\begin{equation}\label{eq:problem}
\begin{split}
    &\min_{x_i \in \R^n, \ i \in [N]} \sum_{i = 1}^N f_{i,k}(x_i) \\
    &\qquad \text{s.t.} \ x_1 = \ldots = x_N
\end{split}
\end{equation}
where the local costs $f_{i,k}$ change over time in response to changes in the environment.

Online optimization has been extensively explored in both centralized \cite{dallanese_optimization_2020,simonetto_timevarying_2020} and distributed scenarios \cite{li_survey_2023}.
Different online distributed algorithms have been proposed and analyzed, including gradient tracking, DGD, ADMM, and mirror descent, to name a few.
Importantly, all these algorithms can only achieve inexact tracking of the optimal solution sequence of~\eqref{eq:problem} \cite{li_survey_2023}. This is proven by a sub-linearly growing regret, or by bounds to the distance from the optimal trajectory \cite{li_survey_2023}. Additionally, gradient tracking methods might be outperformed by DGD \cite{yuan_can_2020}, even though they converge exactly in a static scenario while DGD does not.
Thus in this paper we are interested in designing a class of algorithms which can \textit{achieve exact convergence for online distributed problems}.

To this end, we adopt a control-based design approach, which is proving a fruitful and effective framework \cite{dorfler_systems_2024}.
Control theory has been leveraged to design a broad range of \textit{static} distributed optimization algorithms \cite{wang_control_2010,bin_stability_2022,vanscoy_universal_2022,zhang_frequency_2024,carnevale_unifying_2025}, yielding several gradient tracking methods. However, these algorithms are not tailored for online problems, and thus only converge inexactly.
Additionally, control-based online algorithms have been designed in \cite{bastianello_internal_2024,casti_control_2025,casti_stochastic_2024,simonetto_nonlinear_2024}, leveraging different techniques including internal model principle, robust control, and Kalman filter.
However, these algorithms do not apply directly to a distributed scenario.
Following \cite{simonetto_timevarying_2020}, we refer to static algorithms applied to online problems as \textit{unstructured}, while algorithms tailored to online problems are called \textit{structured}.

Therefore, in this paper we draw from the literature on gradient tracking \cite{alghunaim_decentralized_2021} and the control-based approach of \cite{bastianello_internal_2024,casti_control_2025} to design a novel class of structured online distributed algorithms.
The idea is to leverage the problem reformulation of \cite{alghunaim_decentralized_2021}, which unifies several gradient tracking methods. But instead of solving the reformulated problem with static primal-dual, we solve it with the control-based algorithm of \cite{casti_control_2025}.
The resulting class of distributed algorithms is characterized by the choice of a triplet of consensus matrices.
We prove that the algorithms converge \textit{exactly} to the optimal trajectory when applied to \textit{quadratic problems}.
Designing the algorithm requires knowledge of a model of the problem variation over time. Thus, we also propose a provably convergent procedure to build this model in a distributed fashion.
Finally, we evaluate the algorithms' performance for non-quadratic problems, showing that they outperform unstructured alternatives, even though they do not converge exactly.

\section{Background}\label{sec:background}
In this section we briefly review \cite{alghunaim_decentralized_2021}, which provides a unified framework to design and analyze (static) gradient tracking algorithms, and apply it to the online problem~\eqref{eq:problem}. This unified perspective will then be instrumental in section~\ref{sec:structured} to design our proposed algorithm.

In the following, we model the agents interconnections by the graph $\mathcal{G} = (\mathcal{V}, \mathcal{E})$, $\mathcal{V} = [N]$, assumed undirected and connected.

\subsection{Problem reformulation}\label{subsec:problem-reformulation}
The first step is to reformulate~\eqref{eq:problem} into a saddle point problem.
To this end, we notice that the following problem
\begin{equation}\label{eq:problem-regularized}
\begin{split}
    &\min_{x_i \in \R^n, \ i \in [N]} \ \mu \sum_{i = 1}^N f_{i,k}(x_i) + \frac{1}{2} \x^\top \W_3 \x \\
    &\hspace{2cm} \text{s.t.} \ \W_2 \x = \0
\end{split}
\end{equation}
where $\x = [x_1^\top, \ldots, x_N^\top]^\top \in \R^{n N}$, is equivalent to~\eqref{eq:problem} provided that i) $\W_2 \x = \0$ if and only if $x_1 = \ldots = x_N$; and ii) $\W_3$ is either zero or $\W_3 \x = \0$ if and only if $\W_2 \x = \0$.
$\mu > 0$ is a tunable weight parameter that in section~\ref{subsec:unstructured} will serve as step-size.

Solving~\eqref{eq:problem-regularized} is in turn equivalent to solving the saddle point problem
\begin{equation}\label{eq:problem-saddle}
    \min_{\x, \w} \mathcal{L}_k(\x, \w)
\end{equation}
where we define the (online) Lagrangian
\begin{equation}\label{eq:lagrangian}
    \mathcal{L}_k(\x, \w) = \mu \sum_{i = 1}^N f_{i,k}(x_i) + \w^\top \W_2 \x + \frac{1}{2} \x^\top \W_3 \x,
\end{equation}
and $\w = [w_1^\top, \ldots, w_N^\top]^\top$, $w_i \in \R^n$, are the dual variables.

\subsection{Unstructured algorithm design}\label{subsec:unstructured}
Following \cite{alghunaim_decentralized_2021}, we can now apply the gradient descent-ascent method to~\eqref{eq:problem-saddle} (with step-size of $1$), which yields the updates
\begin{subequations}\label{eq:unstructured}
\begin{align}
    \z_{k+1} &= \x_k - \left( \mu \nabla f_k(\x_k) + \W_2 \w_k + \W_3 \x_k \right) \label{eq:unstructured-primal} \\
    \w_{k+1} &= \w_k + \W_2 \z_k \label{eq:unstructured-dual} \\
    \x_{k+1} &= \W_1 \z_{k+1} \label{eq:unstructured-combine}
\end{align}
\end{subequations}
where $\nabla f_k(\x)$ is the vector stacking all the local gradients $\nabla f_{i,k}(x_i)$, and $\w_0 = \0$.
Clearly, computing $\W_2 \w_k$ and $\W_3 \x_k$ requires a round of peer-to-peer communications; moreover,~\eqref{eq:unstructured-combine} serves as a ``combine'' step, applying an additional consensus round between the agents.

Different choices of the consensus matrices $\W_1$, $\W_2$, $\W_3$ yield different gradient tracking techniques. Defining $\W = W \otimes I_n$, with $W$ a doubly stochastic matrix for the graph $\mathcal{G}$, Table~\ref{tab:matrix-triplets} reports some options \cite{alghunaim_decentralized_2021}.
\begin{table}[!ht]
\caption{Consensus matrices triplets and corresponding gradient tracking algorithms.}
\label{tab:matrix-triplets}
    \centering
    \begin{tabular}{l c c c}
         Algorithm & $\W_1$ & $\W_2^2$ & $\W_3$ \\ 
        \hline
        Aug-DGM  & $\W^2$ & $(I - \W)^2$ & $0$ \\
    
        Exact diffusion & $0.5(I + \W)$ & $0.5(I - \W)$ & $0$ \\
  
        DIGing & $I$ & $(I - \W)^2$ & $I - \W^2$ \\

        EXTRA & $I$ & $0.5(I - \W)$ & $0.5(I - \W)$ \\

    \end{tabular}
\end{table}

The results of \cite{alghunaim_decentralized_2021} guarantee convergence of~\eqref{eq:unstructured} for \textit{static problems}, but not for online ones.
Indeed,~\eqref{eq:unstructured} is an unstructured online algorithm, in that it does not account for the time-variability of the problem in its design. The following section will then fill this gap resorting to a control-based approach.

\section{Algorithm Design and Analysis}\label{sec:structured}

In section~\ref{subsec:problem-reformulation} we have reformulated~\eqref{eq:problem} into a regularized, online problem with linear equality constraints. Therefore, we can apply the control-based approach of \cite{casti_control_2025} to design a suitable algorithm.
To this end, we introduce first the following assumptions.

\begin{assumption}[Cost]\label{as:cost}
The local costs are quadratic:
$$
    f_{i,k}(x) = \frac{1}{2} x^\top A_i x + \langle b_{i,k}, x \rangle, \quad i \in [N], \ \forall k \in \N,
$$
with $\lmin I_n \preceq A_i \preceq \lmax I_n$, $\lmax > 0$.
\end{assumption}

\begin{assumption}[Internal model]\label{as:internal-model}
The sequence of vectors $\{ b_{i,k} \}_{k \in \N}$, $i \in [N]$, has rational $\mathcal{Z}$-transform
$$
    \mathcal{Z}[b_{i,k}] = \frac{B_{i,\mathrm{N}}(z)}{B_\mathrm{D}(z)}
$$
where the polynomial $B_\mathrm{D}(z) = z^m + \sum_{\ell = 0}^{m - 1} p_\ell z^\ell$ is known and has marginally stable zeros.
\end{assumption}

Assumption~\ref{as:cost} is required, in line with \cite{casti_control_2025}, so that the algorithm design can be recast as a linear, robust control problem. However, the resulting algorithm can be applied to general, non-quadratic problems as well, and outperforms unstructured alternatives as reported in section~\ref{subsec:numerical-nonquadratic}.
Assumption~\ref{as:internal-model} is required to achieve exact tracking. For simplicity, we take the denominator $B_\mathrm{D}(z)$ to be the same across all agents; this is not restrictive, as $B_\mathrm{D}(z)$ can be taken as the common denominator of all agents' local transfer functions. In section~\ref{subsec:controller-design} we detail how to compute this common denominator in a distributed fashion.

\subsection{Structured algorithm design}
Assumptions~\ref{as:cost},~\ref{as:internal-model} allow us to apply the algorithm design of \cite{casti_control_2025}, and the result is characterized by the following updates
\begin{subequations}\label{eq:structured-global}
\begin{align}
    \mathbold{\xi}_{k+1} & = \left( F \otimes I_{nN} \right) \mathbold{\xi}_k + \left( G \otimes I_{nN}\right)\nabla_x\mathcal{L}_k\left(\y_k,\uv_k\right) \label{eq:structured-global-xi} \\
    \mathbold{\omega}_{k+1} &= \left(F\otimes I_{nN}\right)\mathbold{\omega}_k + \left( G \otimes I_{nN} \right) \nabla_w \mathcal{L}_k(\y_k,\uv_k) \label{eq:structured-global-om} \\
    \y_{k+1} &= \left( H \otimes I_{nN} \right) \mathbold{\xi}_{k+1} \label{eq:structured-global-y} \\
    \uv_{k+1} &= -\tau \left( H \otimes I_{nN}\right)\mathbold{\omega}_{k+1} \label{eq:structured-global-u} \\
    \x_{k+1} &= \W_1 \y_{k+1} \label{eq:structured-global-x} \\
    \w_{k+1} & = \W_1 \uv_{k+1} \label{eq:structured-global-w}
\end{align}
\end{subequations}
where $\mathbold{\xi}$, $\mathbold{\omega}$ are the states of the internal model, $\y$, $\uv$ are the outputs of the internal model, and $\x$, $\w$ the primal and dual variables after a ``combine'' step.
The online Lagrangian is defined according to~\eqref{eq:lagrangian} and, as before, its computation requires peer-to-peer communications.
Finally, the matrices defining~\eqref{eq:structured-global} are
\begin{align*}
    F &= \begin{bmatrix} 0 & 1 & 0 & \cdots & \\ \vdots  & & \ddots & \\  0 & \cdots & 0 &1\\-p_0 & \cdots & \cdots & p_{m-1}\end{bmatrix},%
    \hspace{0.3cm} G = \begin{bmatrix}
        0 \\ \vdots\\0\\1\end{bmatrix} \\
    H&= \begin{bmatrix}
        h_0 & h_1 & \cdots & \cdots & h_{m-1}\end{bmatrix}.
\end{align*}
where $\{ p_\ell \}_{\ell = 0}^{m-1}$ are the coefficients of the internal model Z-transform (see Assumption~\ref{as:internal-model}), and $\{ h_\ell \}_{\ell = 0}^{m-1}$ the coefficients of the stabilizing controller designed according to \cite{casti_control_2025}.
We remark that we take the Kronecker product of $F$, $G$, $H$ with $I_{nN}$ to account for the fact that we have $N$ local states of dimension $n$.

Figure~\ref{fig:block-diagram} reports a block diagram representation of~\eqref{eq:structured-global}, where $\mathcal{F} = \blk\{ F, F \} \otimes I_{n N}$, $\mathcal{G} = \blk\{ G, G \} \otimes I_{n N}$, $\mathcal{H} = \blk\{ H, H \} \otimes I_{n N}$.

\begin{figure}[!ht]
\centering
\includegraphics[width=0.95\linewidth]{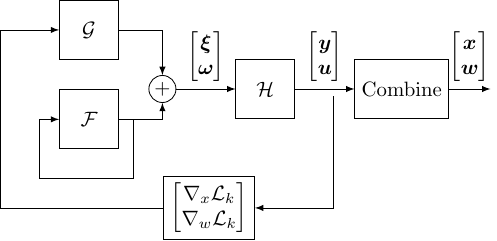}
\caption{Block diagram depiction of~\eqref{eq:structured-global}.}
\label{fig:block-diagram}
\end{figure}

\smallskip

Inspecting~\eqref{eq:structured-global} it is not readily apparent whether the algorithm can indeed be implemented in a distributed fashion.
Exploiting the separable structure of~\eqref{eq:problem-regularized}, we can rewrite~\eqref{eq:structured-global} as the following set of updates performed by each agent $i \in [N]$:
\begin{subequations}\label{eq:structured-local}
\begin{align}
    \xi_{i,k+1} &= \left( F \otimes I_n\right) \xi_{i,k} + \left( G \otimes I_n \right) \left[ \nabla_x\mathcal{L}_k(\y_k,\uv_k)\right]_i \label{eq:structured-local-xi} \\
    \omega_{i,k+1} &= \left( F \otimes I_n\right) \omega_{i,k} + \left( G \otimes I_n \right) \left[ \nabla_x\mathcal{L}_w(\y_k,\uv_k)\right]_i \label{eq:structured-local-om} \\
    y_{i,k+1} &= \left( H \otimes I_n\right)\xi_{i,k+1} \label{eq:structured-local-y} \\
    u_{i,k+1} &= -\tau \left( H \otimes I_n\right) \omega_{i,k+1} \label{eq:structured-local-u} \\
    x_{i,k+1} &= \sum_{j \in \mathcal{N}_i \cup \{i\}}a_{ij} y_{j,k+1} \label{eq:structured-local-x} \\
    w_{i,k+1} &= \sum_{j \in \mathcal{N}_i \cup \{i\}}a_{ij} u_{j,k+1} \label{eq:structured-local-w}
\end{align}
\end{subequations}
where $\xi_i$, $\omega_i$ are the states of the local internal model, $y_i$, $u_i$ are the outputs of the local internal model, and $x_i$, $w_i$ the primal and dual local states.
The $i$-th components of the Lagrangian's gradients are given by
\begin{subequations}\label{eq:local-gradients}
\begin{align}
    \left[\nabla_x \mathcal{L}_k \right]_i &=\mu \nabla f_{i,k}\left( y_{i,k} \right) + \sum_{j \in \mathcal{N}_i \cup \{i\}} \left( c_{ij} y_{j,k} + b_{ij} u_{j,k}\right) \\
    \left[\nabla_w\mathcal{L}_k\right]_i &= \sum_{j \in \mathcal{N}_i \cup \{i\}} b_{ij} u_{j,k}
\end{align}
\end{subequations}
making explicit the need for peer-to-peer communications.

\smallskip

Algorithm~\ref{alg:distributed} reports the pseudo-code implementation of~\eqref{eq:structured-local}, highlighting the necessary operations required by the agents.
\begin{algorithm}[!ht]
\caption{Control-based online distributed optimization}
\label{alg:distributed}
\begin{algorithmic}
\Require{Consensus matrices $\W_1$, $\W_2$, $\W_3$; internal model and controller $F$, $H$; step-size $\mu$ and scaling parameter $\tau$. Initialize $y_{i,0}$ and $u_{i,0} = 0$, $i \in [N]$.}
\For{$k=0,1,2,\ldots$ each agent $i \in [N]$:}
    \CommentState{Communication}
    \State sends $y_{i,k}$ and $u_{i,k}$ to each $j \in \mathcal{N}_i$
    \State and receives $y_{j,k}$ and $u_{j,k}$ from each $j \in \mathcal{N}_i$

    \CommentState{Combine}
    \State updates the local primal and dual states by~\eqref{eq:structured-local-x},~\eqref{eq:structured-local-w}

    \CommentState{Local updates}
    \State observes the cost $f_{i,k}$ and computes the gradients~\eqref{eq:local-gradients}
    \State updates the internal model by~\eqref{eq:structured-local-xi},~\eqref{eq:structured-local-om}
    \State computes the outputs by~\eqref{eq:structured-local-y},~\eqref{eq:structured-local-u}
\EndFor
\end{algorithmic}
\end{algorithm}
The first step is the peer-to-peer transmission of the $y$ and $u$ variables, the only necessary communication step. Then the agents update their local estimates of the primal and dual solutions ($x$, $w$). Finally, they update their local internal models and compute the next outputs $y$ and $u$.

\begin{remark}[Predictive optimization]
We remark that Algorithm~\ref{alg:distributed} works in a predictive fashion, as the costs $f_{i,k}$, $i \in [N]$, are used to compute the states $x_{i,k+1}$ to be applied at time $k+1$ \cite{bastianello_extrapolation_2023}.
\end{remark}

\begin{remark}[Beyond quadratic]\label{rem:non-quadratic}
Algorithm~\ref{alg:distributed} can be applied for a wider range of problems beyond quadratic ones. Indeed, we can see from~\eqref{eq:local-gradients} that the agents only need to access local gradient evaluations in order to apply the algorithm. We evaluate the corresponding performance in section~\ref{sec:numerical}.
\end{remark}

\subsection{Distributed controller design}\label{subsec:controller-design}
Algorithm~\ref{alg:distributed} requires each agent to have access to the internal model $F$ and controller $H$. However, this is global information, and no single agent has access to them.
Thus, in the following we sketch a distributed procedure that the agents can apply to obtain this information.

The first step is for the agents to construct $F$. Let us assume that each agent $i \in [N]$ has access to the local internal model
$$
    \mathcal{Z}[b_{i,k}] = \frac{B_{i,\mathrm{N}}(z)}{B_{i,\mathrm{D}}(z)}
$$
where $B_{i,\mathrm{D}}(z) = z^{m_i} + \sum_{\ell = 0}^{m_i - 1} p_{i,\ell} z^\ell$, with $m_i$ the local degree. The model for example can be constructed from historical observations.
The agents can then apply Algorithm~\ref{alg:common-denominator} to compute the common denominator $B_\mathrm{D}(z)$ of the internal models.
\begin{algorithm}[!ht]
\caption{Distributed computation of $B_\mathrm{D}(z)$}
\label{alg:common-denominator}
\begin{algorithmic}
\Require{Local internal models $B_{i,\mathrm{D}}(z)$, $i \in [N]$; stopping condition $K \in \N$.}

\State{{\color{blue}// Initialization}}
\State Each agent $i \in [N]$ computes the $m_i$ roots $\mathcal{R}_i$ of $B_{i,\mathrm{D}}(z)$

\For{$k=1,2,\ldots, K$ each agent $i \in [N]$:}
    \CommentState{Communication}
    \State sends $\mathcal{R}_i$ to each $j \in \mathcal{N}_i$
    \State and receives $\mathcal{R}_j$ from each $j \in \mathcal{N}_i$

    \CommentState{Combine}
    \State updates $\mathcal{R}_i = \bigcup_{j \in \mathcal{N}_i \cup \{ i \}} \mathcal{R}_j$
\EndFor
\end{algorithmic}
\end{algorithm}
The idea of the algorithm is to distributedly compute the union of the roots of $B_{i,\mathrm{D}}(z)$, $i \in [N]$, which Lemma~\ref{lem:common-denominator} below proves is indeed possible in finite time.
Once the agents have computed $B_\mathrm{D}(z)$, matrix $F$ is simply its canonical control form realization.

\begin{lemma}\label{lem:common-denominator}
Algorithm~\ref{alg:common-denominator} succeeds in computing the common denominator of $B_{i,\mathrm{D}}(z)$, $i \in [N]$ provided that $K$ is larger than the diameter\footnote{That is, the length of the longest path between any two nodes in the graph.} of the graph.
\end{lemma}
\begin{proof}
The graph is assumed to be connected, thus there exists a path between any two nodes, with length smaller than or equal to the diameter $\delta$.
The agents start with $\mathcal{R}_i$ containing the roots of $B_{i,\mathrm{D}}(z)$, and then expand it to include all roots in $\mathcal{R}_j$, $j \in \mathcal{N}_i$ not already included. But it takes at most $\delta$ steps for the roots of $h \in [N]$ to propagate to $i$ by union with the agents in between, and the result holds.
\end{proof}

\smallskip

The second step is for the agents to compute the matrix $H$ characterizing the controller. Given that at this stage the agents already have access to the necessary global information, this step can be performed entirely in parallel.
In particular, following \cite{casti_control_2025}, the controller can be computed solving the following two LMIs w.r.t. $\underline{P}, \overline{P} \succ 0$, $S \in \R^{m \times m}$, and $R \in \R^{1 \times m}$:
\begin{equation}\label{eq:lmi}
\begin{split}
    \begin{bmatrix}
        \underline{P} & F S + \ubar{l} H R\\
        S^{\top} F^{\top} + \ubar{l} R^{\top} H^{\top} & S + S^{\top} - \underline{P}
    \end{bmatrix} &\succ 0,\\
    \begin{bmatrix}
        \overline{P} & F S + \bar{l} H R \\
        S^{\top} F^{\top} + \bar{l} R^{\top} H^{\top} & S + S^{\top} - \overline{P}
    \end{bmatrix} &\succ 0.
\end{split}
\end{equation}
Notice that the values of $\ubar{l}$ and $\bar{l}$ depend on $\underline{\lambda}$, $\overline{\lambda}$, and the choice of consensus matrices -- see \cite{casti_control_2025} for the details.
Given the solution to the LMIs~\eqref{eq:lmi}, the controller is then characterized by $H = R S^{-1}$.

\subsection{Convergence Analysis}
The following result proves the convergence of Algorithm~\ref{alg:distributed} to the optimal solution of~\eqref{eq:problem}.

\begin{proposition}\label{pr:convergence}
Let Assumptions~\ref{as:cost},~\ref{as:internal-model} hold.
Let $\{ \x_{k} \}_{k \to \N}$ be the trajectory generated by Algorithm~\ref{alg:distributed}, and let $\{ x_k^* \}_{k \to \N}$ be the optimal trajectory of~\eqref{eq:problem}.
Then it holds that
$$
    \limsup_{k \to \infty} \norm{\x_k - \x_k^*} = 0.
$$
\end{proposition}
\smallskip
\begin{proof}
Algorithm~\ref{alg:distributed} is designed by applying \cite{casti_control_2025} to~\eqref{eq:problem-regularized}. However, matrix $\W_2$ is not full row rank, and therefore the convergence result of \cite{casti_control_2025} does not apply directly. Indeed, while the primal solution is unique at all times ($\x_k^*$), the dual variable has multiple solutions. In particular, the dual solution of interest is the one lying in the range space of $\W_2$ \cite[Lemma~2]{alghunaim_linear_2020}, that is, $\w_k^* = \W_2 \uv$ for some $\uv \in \R^{n N}$.

Let us consider~\eqref{eq:structured-local} without the combine steps (\eqref{eq:structured-local-x}~\eqref{eq:structured-local-w}).
Then, we can apply \cite[Proposition~2]{bastianello_internal_2024} to prove the convergence of $\y_k$ to $\x_k^*$ and of $\uv_k$ to $\w_k^*$. This is indeed possible because the result guarantee convergence to the subspace of optimal solutions (which is a subspace due to the non-uniqueness of the dual solution).
However, to guarantee that $\uv_k$ convergence to the right dual solution, we need to guarantee that $\uv_k$ lies in the range space of $\W_2$. This is the case by having selected $\uv_0$ \cite{alghunaim_linear_2020}.

Finally, we need to guarantee that also $\x_k = \W_1 \y_k$ converges to $\x_k^*$. But this is implied by the fact that $\W_1$ is a consensus matrix, and that $\x_k^* = \W_1 \x_k^*$.
\end{proof}

\section{Numerical Results}\label{sec:numerical}
In this section we consider the numerical performance of Algorithm~\ref{alg:distributed}. To this end, we consider quadratic costs, as introduced in Assumption \ref{as:cost}, with $n = 15$. The costs have $A_i = V\Lambda V^\top$, where $V$ is a randomly generated orthogonal matrix and $\Lambda$ is diagonal with elements in $[1,5]$.  Finally, we consider a circle graph with $N = 10$ agents.

To avoid computing all optimal solutions, a proxy metric of optimality is used, namely: 
$$
    \varepsilon_k = \norm{\sum_{i = 1}^N \nabla f_{i,k}(\bar{x}_k)}^2 \qquad \bar{x}_k = \frac{1}{N} x_{i,k},
$$
which is zero if the agents have converged to the optimal trajectory of~\eqref{eq:problem}.

\subsection{Choosing the consensus matrices}\label{subsec:cons_mat}
We start by considering local costs characterized by $\{b_{i,k} \}_{k \in \N} = \sin(\nu k )\1$, $i \in [N]$. In this case we have access to an exact internal model. In this setting, we compare some different consensus matrices, as introduced in Table~\ref{tab:matrix-triplets}.

In Figures~\ref{fig:triplets-1} and~\ref{fig:triplets-2} we present the evolution of $\varepsilon_k$ for the different choices of consensus matrices, both for Algorithm~\ref{alg:distributed} (labeled ``Control-based'') and the corresponding gradient tracking method from \cite{alghunaim_decentralized_2021} (labeled ``unstructured'').
We remark that the algorithms in Figure~\ref{fig:triplets-1} apply a combine step ($\W_1 \neq I$), while the algorithms in Figure~\ref{fig:triplets-2} do not.
The step-sizes are hand-tuned for optimal performance.

Clearly, the proposed algorithms converge to the optimal trajectory, up to numerical precision, validating Proposition~\ref{pr:convergence}. On the other hand, the unstructured algorithms converge inexactly, generally to a quite large steady-state error.
Additionally, we remark that the choice of triplets affects the rate of convergence during the initial transient; notice in particular the different scales of the x-axes.

\begin{figure}[!ht]
    \includegraphics[width=0.95\linewidth]{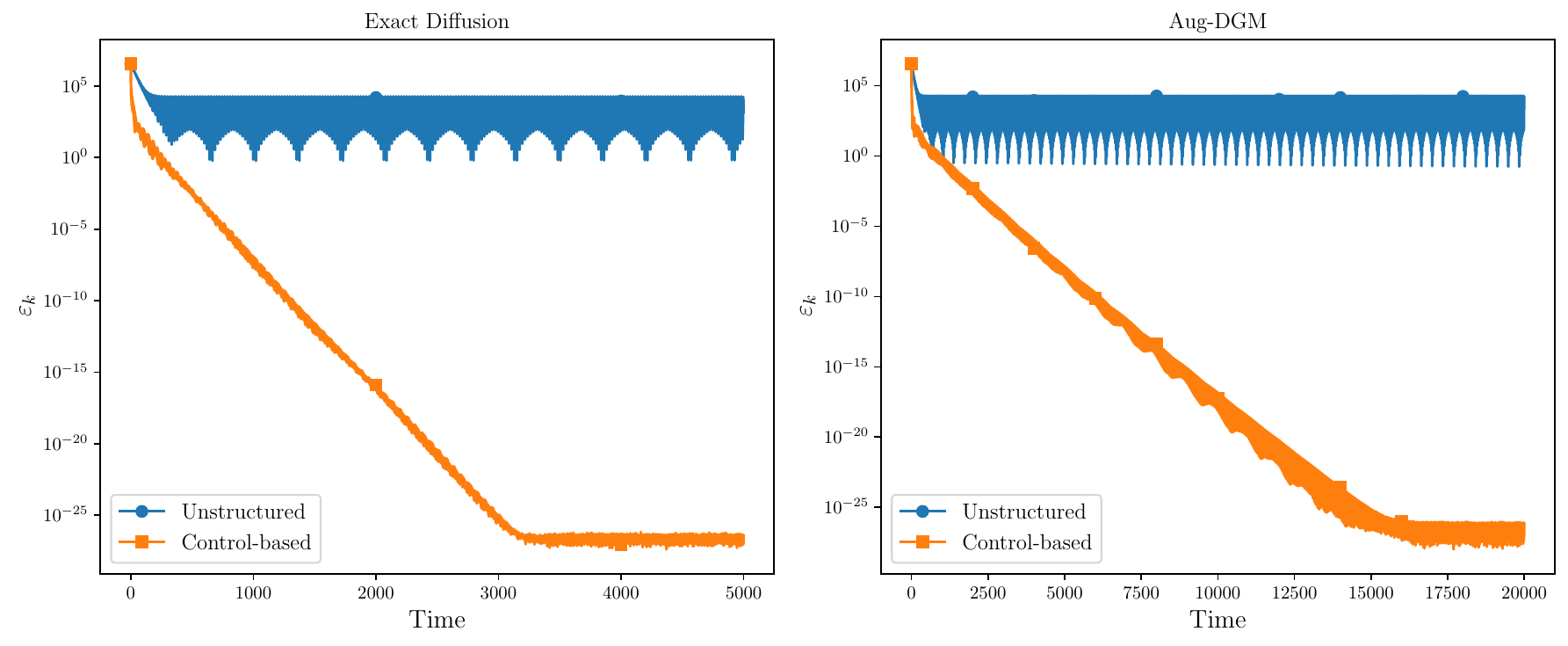}
    \caption{Tracking error comparison for consensus matrices with $\W_1 \neq I$ }
    \label{fig:triplets-1}
\end{figure}

\begin{figure}[!ht]
    \includegraphics[width=0.95\linewidth]{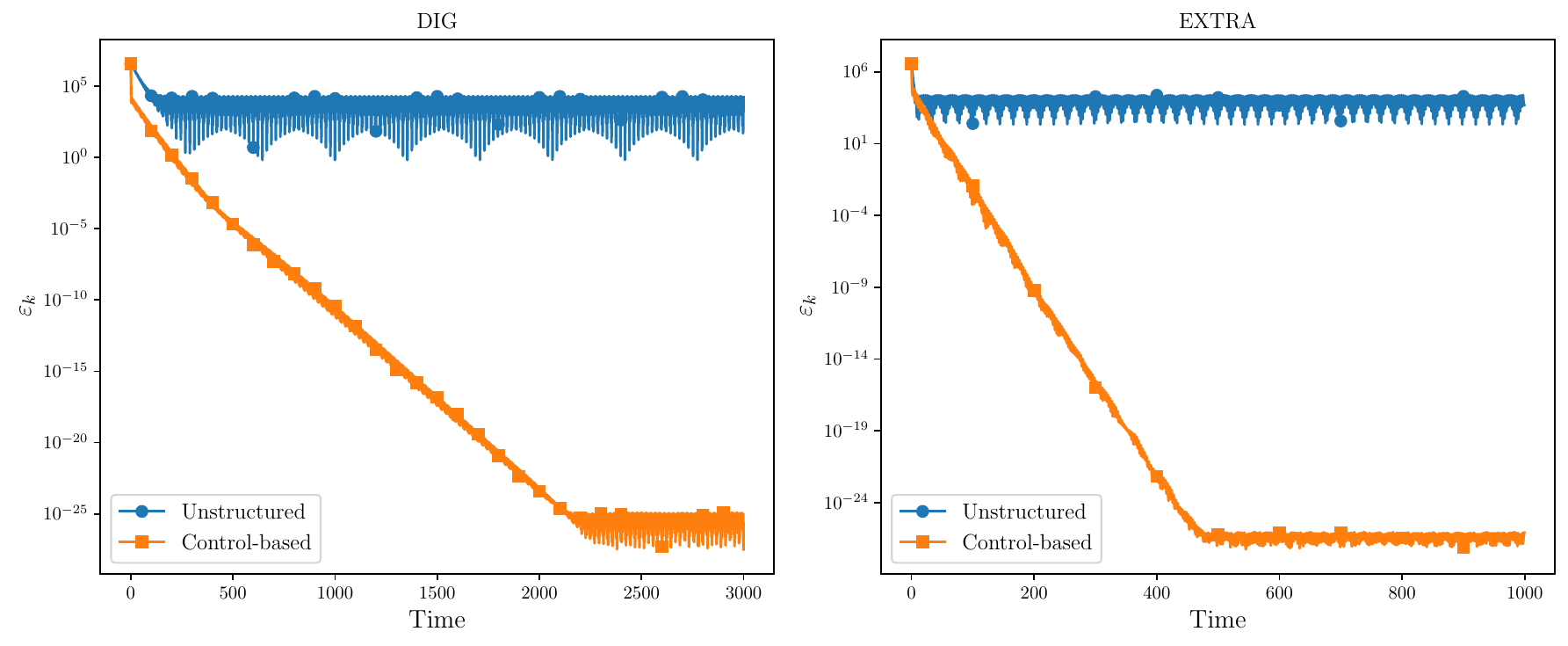}
        \caption{Tracking error comparison for consensus matrices with $\W_1 = I$ }
    \label{fig:triplets-2}
\end{figure}

\subsection{Changing signals}\label{subsec:changing_signals}
We turn now to evaluating the effect of different internal models on the algorithms of interest. For simplicity, we focus on the consensus matrices triplet DIGing.
We apply the algorithm to three different problems, characterized by linear terms $\{b_{i,k} \}_{k \in \N} $ that change as: $(i)$ ramp: $b_{i,k}  = k \bar{b}, \bar{b} \in \mathbb{R}^n $, $(ii)$ sine $b_{i,k}  =\sin(\nu k )\1$ (as in sec.~\ref{subsec:cons_mat}), and $(iii)$ squared sine $b_{i,k} =\sin^2(\nu k )\1$.

\begin{figure}[!ht]
    \centering
    \includegraphics[width=0.9\linewidth]{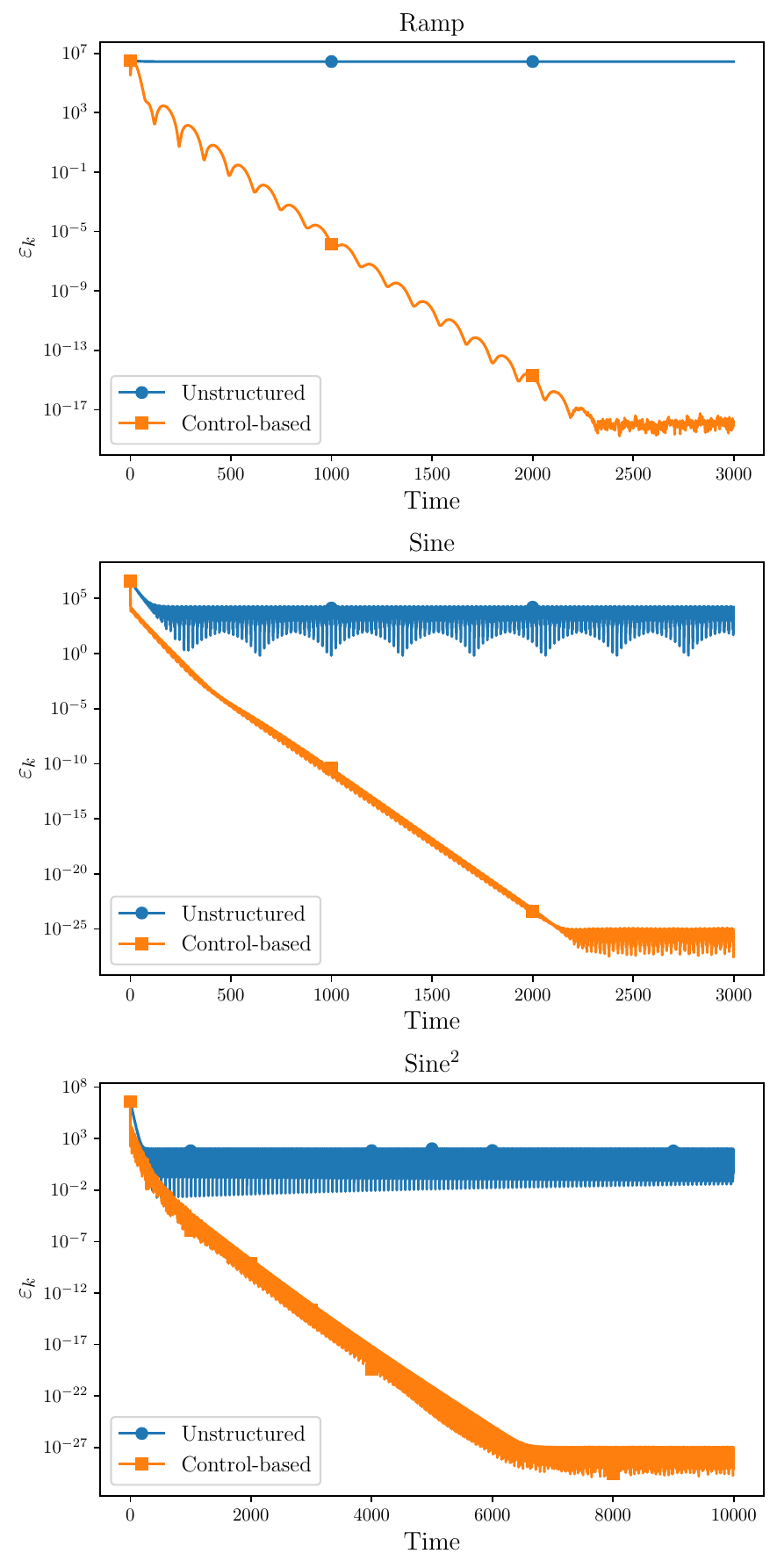}
    \caption{Tracking error comparison with different $b_{i,k} $}
    \label{fig:changing_signals}
\end{figure}

As we can see in Figure~\ref{fig:changing_signals}, for each of the different signal types, the $\varepsilon_k$ of the proposed structured algorithm converges to zero (up to numerical precision), further validating Proposition~\ref{pr:convergence}.         
As before, it can be observed that the unstructured algorithm performs many orders of magnitude worse.
We also notice that the rate of convergence during the initial transient is affected by the internal model. In particular, while in the ramp and sine scenarios the rate is similar, the higher order sine$^2$ internal model achieves slower convergence.

\subsection{Inexact internal model}
In the previous sections we assumed the agents had access to a precise internal model. However, in practice this might not be possible, especially when estimating the model from data.
In this section we evaluate the performance of the proposed algorithm when only an inexact internal model is available. In particular, we consider case $(ii)$ of section~\ref{subsec:changing_signals}, which is characterized by the internal models $B_{i,\mathrm{D}}(z) = z^2 - 2\cos(\nu)z - 1$. We then run the proposed algorithm with a perturbed internal model where $\hat{p}_1 = 2\cos(\nu) + e$, where the perturbation $e$ is drawn from $[0,0.1]$.

\begin{figure}[!ht]
    \includegraphics[width=0.95\linewidth]{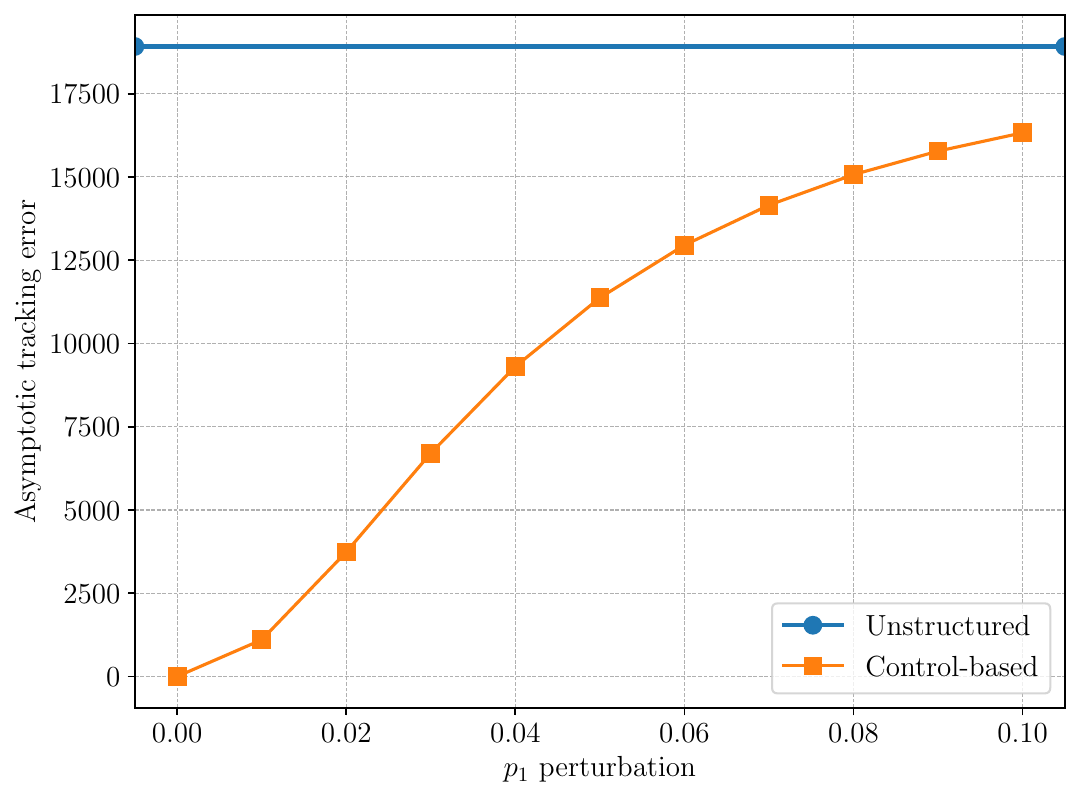}
    \caption{Asymptotic error when a perturbation is applied to $p_1$}
    \label{fig:inexact_model}
\end{figure}

Figure~\ref{fig:inexact_model} depicts the \textit{asymptotic tracking error}\footnote{Computed as the maximum $\varepsilon_k$ over the last 4/5 of the simulation.} attained by the algorithm for different choices of model perturbation $e$.
As we can see, the proposed algorithm does not converge exactly when the model is perturbed, and the asymptotic error grows as the perturbation grows.
However, the trend is not linear, since for larger perturbations, the impact on the asymptotic error grows less severe. Nevertheless, even for the largest perturbations, the proposed algorithm outperforms the unstructured alternative. 

\subsection{Non-quadratic}\label{subsec:numerical-nonquadratic}
As discussed in Remark~\ref{rem:non-quadratic}, the algorithm can also be applied in a non-quadratic scenario, and in this section we evaluate its performance in this context.
In particular, we now consider local cost functions of the form: \begin{equation}\label{eq:nonquad-cost}
        f_{i,k}(x) = \frac{1}{2} x^\top A_i x + \langle b_{i,k}, x \rangle + \sin(\nu k)\log[1+\exp (c^\top x)],
\end{equation}
where $i \in [N], \ \forall k \in \N$, and with $A_i = V \Lambda V^\top$, where $\Lambda$ has elements between $[1,10]$. Furthermore for this specific example, $\nu = 5$ and $b_{i,k}, c \in \R^n$, with $b_{i,k}$ constant and $\norm{c} = 1$. We keep the dimensions $n = 15$ and number of agents $N = 10$ as before. 

In order to apply the algorithm, we need to specify an internal model which captures the evolution of $ f_{i,k}(x)$. Due to the periodic nature of $ f_{i,k}(x)$, reasonable estimates for the internal model may be defined as follows: 
\begin{equation}\label{eq:approx_internal_model}
B_{i,\mathrm{D}}(z) = (z-1) \prod_{\ell=1}^L(z^2-2\cos(\ell\nu)z+1), \ L = 1,2,3 .
\end{equation}

\begin{figure}[!ht]
    \includegraphics[width=0.95\linewidth]{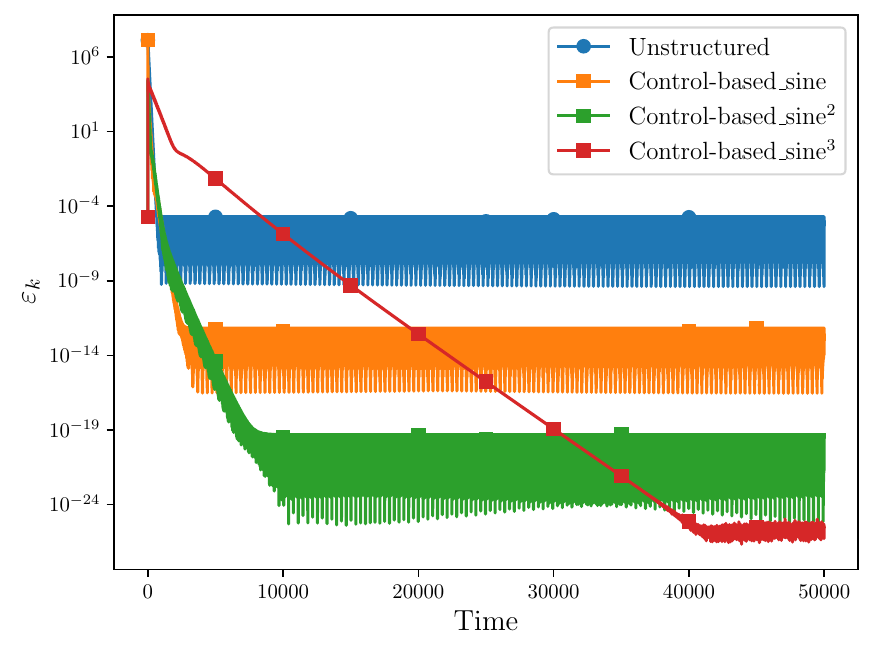}
    \caption{Tracking error for the proposed algorithm applied to~\eqref{eq:nonquad-cost} with~\eqref{eq:approx_internal_model}.}
    \label{fig:non-quadratic}
\end{figure}

Figure~\ref{fig:non-quadratic} compares the convergence of Algorithm~\ref{alg:distributed} against that of the unstructured alternative.
Clearly, Algorithm~\ref{alg:distributed} does not converge exactly because the assumptions of Proposition~\ref{pr:convergence} are not verified. However, the result again demonstrates that the proposed control-based algorithm greatly outperforms the unstructured algorithm.
As noted in subsection~\ref{subsec:changing_signals}, higher internal model orders lead to slower convergence, however in this case, the slower convergence also results in a lower error. This indicates that approximate internal models can be chosen to converge quickly but with less accuracy, or slowly with greater accuracy.

\section{Conclusions}
In this paper we designed a novel class of online distributed optimization algorithms leveraging control theoretical techniques.
We started by focusing on quadratic costs, and assuming to know an internal model of their variation. In this set-up, we formulated the algorithm design as a robust control problem, showing that it yields a fully distributed algorithm. We also provided a distributed routine to acquire the internal model.
We showed that the algorithm converges exactly to the sequence of optimal solutions.
We empirically evaluated the performance of the algorithm for different choices of parameters.
Additionally, we evaluated the performance of the algorithm for quadratic problems with inexact internal model and non-quadratic problems, and showed that it outperforms alternative algorithms in both scenarios.
Future research directions include exploring the design of control-based algorithms that achieve exact convergence for non-quadratic problems as well.

\bibliographystyle{IEEEtran}
\bibliography{references}

\end{document}